\documentclass[10pt,a4paper,reqno]{amsart}
\usepackage{amsmath,amssymb, amsbsy}
\usepackage{color,psfrag}
\usepackage[dvips]{graphicx}
\usepackage{color}
\theoremstyle{plain}
\newtheorem{thm}{Theorem}[section]
\theoremstyle{plain}
\newtheorem{lem}[thm]{Lemma}

\newtheorem{cor}[thm]{Corollary}

\theoremstyle{definition}

\newtheorem{rem}{Remark}[section]

\newcommand{\hn}{\mathbb{H}^{N}}

\newcommand{\authorfootnotes}{\renewcommand\thefootnote{\@fnsymbol\c@footnote}}%

\numberwithin{equation}{section} \allowdisplaybreaks

\begin{document}
        \title[]{Improved higher order poincar\'e inequalities\\ on the hyperbolic space via Hardy-type remainder terms}

\date{}

\author[Elvise BERCHIO]{Elvise BERCHIO}
\address{\hbox{\parbox{5.7in}{\medskip\noindent{Dipartimento di Scienze Matematiche, \\
Politecnico di Torino,\\
        Corso Duca degli Abruzzi 24, 10129 Torino, Italy. \\[3pt]
        \em{E-mail address: }{\tt elvise.berchio@polito.it}}}}}
\author[Debdip GANGULY]{Debdip GANGULY}
\address{\hbox{\parbox{5.7in}{\medskip\noindent{Dipartimento di Scienze Matematiche, \\
Politecnico di Torino,\\
        Corso Duca degli Abruzzi 24, 10129 Torino, Italy. \\[3pt]
        \em{E-mail address: }{\tt debdip.ganguly@polito.it}}}}}
 
\date{\today}



\begin{abstract}
The paper deals about Hardy-type inequalities associated with the following higher order Poincar\'e inequality:
   
   \[
   \left( \frac{N-1}{2} \right)^{2(k -l)} := \inf_{ u \in C_{c}^{\infty} \setminus \{0\}} \frac{\int_{\hn} |\nabla_{\hn}^{k} u|^2  \ dv_{\hn}}{\int_{\hn} |\nabla_{\hn}^{l} u|^2  \ dv_{\hn} }\,,
   \]
where $0 \leq l < k$ are integers and $\hn$ denotes the hyperbolic space. More precisely, we improve the Poincar\'e inequality associated with the above ratio by showing the existence of $k$ Hardy-type remainder terms. Furthermore, when $k = 2$ and $l = 1$ the existence of further remainder terms are provided and the sharpness of some constants is also discussed. As an application, we derive improved Rellich type inequalities on upper half space of the Euclidean space with non-standard remainder terms.
\end{abstract}

\maketitle

 \section{Introduction}
Let $\hn$ denote the hyperbolic space and let $k, l$ be non-negative integers such that $l < k$. The following higher order Poincar\'e inequality \cite[Lemma 2.4]{SD} holds
 \begin{equation}\label{high}
\int_{\mathbb{H}^{N}} |\nabla_{\mathbb{H}^{N}}^{k} u|^{2}  \ dv_{\mathbb{H}^{N}} \geq \left( \frac{N-1}{2} \right)^{2(k - l)} 
\int_{\mathbb{H}^{N}} |\nabla_{\mathbb{H}^{N}}^{l} u|^2 \ dv_{\mathbb{H}^{N}},
\end{equation}
for all $u \in H^{k}(\mathbb{H}^{N})$, where  
$$ \nabla_{\mathbb{H}^{N}}^{j}:=
\left \{\begin{array}{ll} 
\Delta_{\hn}^{j/2} & \text{if $j$ is an even integer}\,,\\
\nabla_{\hn} \Delta_{\hn}^{(j-1)/2} & \text{if $j$ is an odd integer}
 \end{array}\right.$$
and $\nabla_{\mathbb{H}^{N}}$ denotes the Riemannian gradient while $\Delta_{\hn}^j$ denotes the $j-$th iterated Laplace-Beltrami operator.  
The present paper takes the origin from the basic observation that the inequality in \eqref{high} is strict for $u\neq 0$, namely the following infimum is never achieved
\begin{equation*}\label{inf}
  \left(  \frac{N-1}{2}    \right)^{2(k-l)} = \inf_{u \in H^{k}(\mathbb{H}^{N}) \setminus \{ 0 \} }
\frac{\int_{\mathbb{H}^{N}} |\nabla_{\mathbb{H}^{N}}^{k} u|^{2}  \ dv_{\mathbb{H}^{N}}}{\int_{\mathbb{H}^{N}} |\nabla_{\mathbb{H}^{N}}^{l} u|^2 \ dv_{\mathbb{H}^{N}}}\,.
\end{equation*}
It becomes then a natural problem to look for possible remainder terms for \eqref{high}. In this direction, when $k=1$ and $l=0$, a remainder term of Sobolev type has been determined in \cite{mancini}. The aim of our study is to deal with Hardy remainder terms, namely to determine improved Hardy inequalities for higher order operators, where the improvement is meant with respect to the higher order Poincar\'e inequality \eqref{high}. More precisely, settled $r:=\varrho(x,x_0)$, where $\varrho$ denotes the geodesic distance and $x_0\in{\mathbb H}^N$ denotes the pole, we wish to answer the question   
\\

\noindent \it Does there exist positive constants $C$ and $\gamma$ such that the following Poincar\'e-Hardy inequality
\begin{equation}\label{motivation}
\int_{\mathbb{H}^{N}} |\nabla_{\mathbb{H}^{N}}^{k} u|^{2}  \ dv_{\mathbb{H}^{N}} - \left( \frac{N-1}{2} \right)^{2(k - l)} 
\int_{\mathbb{H}^{N}} |\nabla_{\mathbb{H}^{N}}^{l} u|^2 \ dv_{\mathbb{H}^{N}}\ge C \int_{{\mathbb H}^N}\,\frac{u^2}{r^{\gamma}}\,{\rm d}v_{\hn}\,
\end{equation}
holds for all $u\in H^k(\hn)$?
\vspace{.3truecm}
\rm
\\

The literature on improved Hardy and Rellich inequalities in the Euclidean setting dates back to the seminal works of Brezis-Vazquez \cite{Brezis} and  Brezis-Marcus \cite{BrezisM}. Without claiming of completeness, we also recall \cite{A,GFT,BFT2,BT,DH,FT,FTT,gaz,GM,MMP,R,TZ} and references therein. The reason of such a great interest is surely do to the fact that Hardy inequalities and their improved versions have various applications in the theory of partial differential equations and nonlinear analysis, see for istance \cite{Brezis,vaz,VZ}. Further generalizations to Riemannian manifolds are quite recent and a subject of intense research after the work of Carron \cite{Carron}. We enlist few important recent works \cite{Mitidieri1,Mitidieri2,Dambrosio,pinch,Kombe1,Kombe2,LW,Mitidieri,Yang} and references therein. Most of these works deals with classical Hardy inequalities and their improvement on Riemannian manifolds. Namely, differently from \eqref{motivation}, the optimal Hardy constant is taken as fixed and one looks for bounds of the constant in front of other remainder terms. The main motivation of our study initiated in \cite{BGG} on improved Poincar\'e inequalities comes from a paper of Devyver-Fraas-Pinchover \cite{pinch}, which deals with optimal Hardy inequalities for general second order operators. In particular, the existence of at least one Hardy-type remainder term for \eqref{high} with $k=1$ and $l=0$ follows as an application of their results. Nevertheless, their weight is given in terms of the Green's function of the associated operator and does not imply the validity of an inequality like \eqref{motivation}. See \cite{BGG} for further details. The same can be said for the inequality in \cite[Example 5.3]{BMR} where $N=3$. The above mentioned goal was achieved in \cite{BGG} where, developing a suitable construction of super solution, the following inequality was shown
\par \bigskip\par
 {$\bullet$ \bf{Case $k = 1$ and $l=0$.}} For $N>2$ and for all $ u \in C^{\infty}_{0}(\mathbb{H}^{N} )$ there holds 
 \begin{equation}\label{poincareeq} 
\int_{\mathbb{H}^{N}} |\nabla_{\hn} u|^2 \ dv_{\hn}-  \left( \frac{N-1}{2} \right)^{2} \int_{\mathbb{H}^{N}} u^2 \ dv_{\hn}\,
\geq \frac{1}{4} \int_{\mathbb{H}^{N}} \frac{u^2}{r^2} \ dv_{\hn}\,,
\end{equation}
where the constants $\left( \frac{N-1}{2} \right)^{2}$ and $\frac{1}{4}$ are sharp.
\par \medskip\par
Unfortunately, the super solution construction applied in the proof of \eqref{poincareeq} seems not applicable to the higher order case. Nevertheless, by exploiting a completely different technique based on spherical harmonics, in \cite{BGG} the following second order analogue of \eqref{poincareeq} was obtained

\par \bigskip\par

 {\bf{$\bullet$ Case $k = 2$ and $l=0$.}} For $N>4$ and for all $ u \in C^{\infty}_{0}(\mathbb{H}^{N})$ there holds

\begin{align}\label{PR} 
 \int_{\mathbb{H}^{N}} (\Delta_{\hn} u)^2  \ dv_{\hn} -  \left( \frac{N-1}{2} \right)^{4}\int_{\hn} u^2 \ dv_{\hn} 
\geq \frac{(N-1)^2}{8} \int_{\mathbb{H}^{N}} \frac{u^2}{r^2} \ dv_{\hn} + \frac{9}{16} \int_{\mathbb{H}^{N}} \frac{u^2}{r^4} \ dv_{\hn} \,,
\end{align}
where the constant $\left( \frac{N-1}{2} \right)^{4}$ and $\frac{(N-1)^2}{8}$ are sharp.

\par \medskip\par
It is clear that \eqref{poincareeq}  and \eqref{PR}  do not give a complete proof of \eqref{motivation}. The aim of the present paper is either to generalize to the higher order \eqref{poincareeq}  and \eqref{PR} and to investigate all the remaining cases when $l\neq 0$. A first step in this direction is represented by the proof of the validity of \eqref{motivation} when $k = 2$ and $ l= 1$. This case is not covered by \eqref{poincareeq} and \eqref{PR} and its proof requires some effort. A clever transformation which uncovers the Poincar\'e term and spherical harmonics technique are the main tools applied, see Sections \ref{21}  and \ref {proof21}. Also we note that when $k = 2$ and $ l= 1$ further singular remainder terms, involving hyperbolic functions, are provided and some optimality issues are proved. Namely, we have  

\begin{thm} \label{main_intro0}
 { \bf{(Case $k = 2$ and $l=1$)}} 
Let $ N >4$. For all $ u \in C^{\infty}_{0}(\mathbb{H}^{N})$ there holds

$$
\int_{\mathbb{H}^{N}} (\Delta_{\hn} u)^2 \ dv_{\hn} - \left( \frac{N-1}{2} \right)^{2} \int_{\mathbb{H}^{N}} |\nabla_{\hn} u|^2 \ dv_{\hn} \geq  \frac{(N-1)^2}{16} \int_{\mathbb{H}^{N}} \frac{u^2}{r^2} \ dv_{\hn} 
 + \frac{9}{16} \int_{\mathbb{H}^{N}} \frac{u^2}{r^4} \ dv_{\hn} $$
 $$+\frac{(N-1)(N-3)(N^2 -2N - 7)}{16} \int_{\hn} \frac{u^2}{\sinh^2 r} \  dv_{\hn} +\frac{(N-1)(N-3)(N^2-4N-3)}{16} \int_{\hn} \frac{u^2}{\sinh^4 r} \ dv_{\hn}.
$$
The constant $\left( \frac{N-1}{2} \right)^{2}$ is sharp by construction and sharpness of the other constants is discussed in Section 2.  
\end{thm}

Theorem \ref{main_intro0} turns out to be one of the key ingredients in our strategy to get the arbitrary case, i.e. inequality \eqref{motivation} for every $l <k$. Furthermore, from Theorem \ref{main_intro0} we derive improved Rellich type inequalities on upper half space of the Euclidean space having their own interest. See Corollary \ref{cor2} for the details. The technique adopted relies on the so-called \lq\lq Conformal Transformation" to the Euclidean space. 
\par \medskip\par
As concerns the general case $l <k$, a fine combination of the previous results and some technical inequalities allow us to finally derive the following family of inequalities
\begin{thm}\label{main_intro}  { \bf{(Case $0\leq l < k$)}} 
Let $k,l$ be integers such that $0\leq l < k$ and let $ N> 2k$. There exist $k$ \emph{positive} constants $\alpha_{k,l}^j=\alpha_{k,l}^j(N)$ such that the following inequality holds

 $$
\int_{\mathbb{H}^{N}} |\nabla_{\mathbb{H}^{N}}^{k} u|^{2}  \ dv_{\mathbb{H}^{N}} -\left( \frac{N-1}{2} \right)^{2(k - l)} 
\int_{\mathbb{H}^{N}} | \nabla^{l}_{\mathbb{H}^{N}} u|^{2} \ dv_{\mathbb{H}^{N}}\ge \sum_{j = 1}^{k} \alpha_{k,l}^j \int_{\mathbb{H}^{N}} \frac{u^2}{r^{2j}} \ dv_{\mathbb{H}^{N}}
$$
for all $u \in C^{\infty}_{0}(\mathbb{H}^{N})$. Furthermore, the constant $\left( \frac{N-1}{2} \right)^{2(k - l)} $ is sharp and the leading terms as $r\rightarrow 0$ and  $r\rightarrow +\infty$, namely $\alpha_{k,l}^{1}$ and $\alpha_{k,l}^{k}$, are given explicitly in Theorems \ref{mainhigher0} and \ref{mainmain} below. 
 
 \end{thm}
In view of possible applications to differential equations, we point out that the strategy of our proofs basically allows to determine explicitly all the constants $\alpha_{k,l}^j$ in Theorem \ref{main_intro}. Nevertheless, for the sake of simplicity, we prefer to focus on the leading terms $\alpha_{k,l}^{1}$ and $\alpha_{k,l}^{k}$. This choice is also justified by the fact that our interest is devoted to the non-Euclidean behavior of inequalities and the constant highlighting this aspect is exactly $\alpha_{k,l}^{1}$, i.e. the constant in front of the leading term as $r\rightarrow +\infty$. As a matter of example, here below we specify our family of inequalities for some particular choices of $k$ and $l$.
 \begin{cor}\label{mainhigher0} { \bf{(Case $0= l < k$)}} 
Let $k$ be a positive integer and let $ N> 2k$.\par
If $k=2m$ for some positive integer $m$, there holds 
$$
\int_{\mathbb{H}^{N}}  (\Delta_{\hn}^{m} u)^2 \ dv_{\mathbb{H}^{N}} -\left( \frac{N-1}{2} \right)^{4m } 
\int_{\mathbb{H}^{N}} u^2 \ dv_{\mathbb{H}^{N}}\geq $$
$$ \sum_{j = 1}^{ m} \frac{(N-1)^{4m-2j}}{2^{4m-1}} \int_{\mathbb{H}^{N}} \frac{u^2}{r^{2}} \ dv_{\mathbb{H}^{N}}+ \frac{9}{2^{4m}} \prod_{j = 1}^{ m -1} (N+ 4j)^2(N-4j-4)^2 \int_{\mathbb{H}^{N}} \frac{u^2}{r^{4m}} \ dv_{\mathbb{H}^{N}}
 $$
for all $u \in C^{\infty}_{0}(\mathbb{H}^{N})$, where we use the convention $\prod_{j = 1}^{ 0}=1$.\par
 If $k=2m+1$ for some positive integer $m$, there holds
 $$
\int_{\mathbb{H}^{N}}  |\nabla_{\hn} (\Delta_{\hn}^{m} u)|^{2}  \ dv_{\mathbb{H}^{N}}  -\left( \frac{N-1}{2} \right)^{4m+2 } 
\int_{\mathbb{H}^{N}} u^2 \ dv_{\mathbb{H}^{N}}\geq $$
$$ \left[ \sum_{j = 1}^{ m} \frac{(N-1)^{4m-2j+2}}{2^{4m+1}}+ \frac{(N-1)^{2m}}{2^{4m+2}} \right]\int_{\mathbb{H}^{N}} \frac{u^2}{r^{2}} 
\ dv_{\mathbb{H}^{N}}+ \frac{1}{2^{4m+2}} \prod_{j = 1}^{ m} (N+ 4j - 2)^2(N- 4j - 2)^2 \int_{\mathbb{H}^{N}} \frac{u^2}{r^{4m+2}} \ dv_{\mathbb{H}^{N}}
 $$
for all $u \in C^{\infty}_{0}(\mathbb{H}^{N})$.

 \end{cor}
 
 \begin{cor}
  { \bf{(Case $k-1= l < k$)}} 
Let $k$ be a positive integer and let $ N> 2k$. \par
 If $k=2m$ for some positive integer $m$, there holds
$$
\int_{\mathbb{H}^{N}}  (\Delta_{\hn}^{m} u)^2 \ dv_{\mathbb{H}^{N}} - \left( \frac{N-1}{2} \right)^{2} \int_{\mathbb{H}^{N}} |\nabla_{\hn} (\Delta_{\hn}^{m-1} u)|^{2} \ dv_{\mathbb{H}^{N}}  \geq $$
$$\frac{(N-1)^{2m}}{2^{4m}} \int_{\mathbb{H}^{N}} \frac{u^2}{r^{2}} \ dv_{\mathbb{H}^{N}} + \frac{9}{2^{4m}} \prod_{j = 1}^{ m-1} ((N+ 4j )(N- 4j -4))^2 \int_{\mathbb{H}^{N}} \frac{u^2}{r^{4m}} \ dv_{\mathbb{H}^{N}}
$$
for all $u \in C^{\infty}_{0}(\mathbb{H}^{N})$, where we use the convention $\prod_{j = 1}^{ 0}=1$.
\par
 If $k=2m+1$ for some positive integer $m$, there holds
$$
\int_{\mathbb{H}^{N}}  |\nabla_{\hn} (\Delta_{\hn}^{m} u)|^{2}  \ dv_{\mathbb{H}^{N}} - \left( \frac{N-1}{2} \right)^{2} \int_{\mathbb{H}^{N}}  (\Delta_{\hn}^{m} u)^{2} \ dv_{\mathbb{H}^{N}}  \geq $$
$$\frac{(N-1)^{2m}}{2^{4m+2}} \int_{\mathbb{H}^{N}} \frac{u^2}{r^{2}} \ dv_{\mathbb{H}^{N}} + \frac{1}{2^{4m+2}} \prod_{j = 1}^{m}(N+ 4j-2)^2(N - 4j- 2)^2 \int_{\mathbb{H}^{N}} \frac{u^2}{r^{4m+2}} \ dv_{\mathbb{H}^{N}} $$
for all $u \in C^{\infty}_{0}(\mathbb{H}^{N})$.
 \end{cor}

The article is organized as follows. Section \ref{21} is devoted to the precise statement and discussion of results for the case $k = 2$ and $ l = 1.$ The complete proof of the results discussed in Section \ref{21} is postponed to Section \ref{proof21}. Section \ref{l0} and Section \ref{karbitrary} are devoted to discussions and proofs of the results for $0= l < k$ and for $0\neq l < k$. The statements of the results given in these sections will contain the precise constants for the leading terms mentioned in the statement of Theorem \ref{main_intro}.

\section{Case $k=2$ and $l=1$}\label{21}
We start by restating Theorem \ref{main_intro0} in its complete form. The proof of the results given in this section will be postponed to Section \ref{proof21}.
 \begin{thm}\label{PRHinequality}
 Let $ N >4$. For all $ u \in C^{\infty}_{0}(\mathbb{H}^{N})$ there holds

\begin{align}\label{npoincare}
\int_{\mathbb{H}^{N}} (\Delta_{\hn} u)^{2} \ dv_{\hn} - \left( \frac{N-1}{2} \right)^{2} \int_{\mathbb{H}^{N}} |\nabla_{\hn} u|^2 \ dv_{\hn} 
 &\geq  \frac{(N-1)^2}{16} \int_{\mathbb{H}^{N}} \frac{u^2}{r^2} \ dv_{\hn} 
 + \frac{9}{16} \int_{\mathbb{H}^{N}} \frac{u^2}{r^4} \ dv_{\hn} \notag \\
& +\frac{(N-1)(N-3)(N^2 -2N - 7)}{16} \int_{\hn} \frac{u^2}{\sinh^2 r} \  dv_{\hn} \notag \\ &+\frac{(N-1)(N-3)(N^2-4N-3)}{16} \int_{\hn} \frac{u^2}{\sinh^4 r} \ dv_{\hn}\,.
\end{align}

The constant $\left( \frac{N-1}{2} \right)^{2}$ is sharp by construction, namely cannot be replaced by a larger one. Furthermore, the constant $\frac{(N-1)^2}{16}$ is sharp in the sense that no inequality of the form
\[
\int_{\hn}   (\Delta_{\hn} u)^{2} \ dv_{\hn} - \left( \frac{N-1}{2} \right)^{2} \int_{\mathbb{H}^{N}} |\nabla_{\hn} u|^2 \ dv_{\hn} 
\geq c\, \int_{\hn} \frac{u^2}{r^2} \ dv_{\hn}
\]
holds for all $ u \in C^{\infty}_{c}(\mathbb{H}^{N})$ when $c>\frac{(N-1)^2}{16}$.
 
 \end{thm}
 
 \begin{rem}
Inequality \eqref{npoincare} does not follow directly from \eqref{poincareeq} and \eqref{PR} but requires an independent proof which is achieved by means of a suitable modification of the proof of \eqref{PR} as given in \cite{BGG}. As already remarked in the Introduction, the main tools exploited are a suitable transformation which uncovers the Poincar\'e term and spherical harmonic analysis. Recently, spherical harmonics technique has been successfully exploited in the context of Weighted Calder\'on–-Zygmund and Rellich inequalities \cite{MSS}.
  \end{rem}
 
 \begin{rem}
As already explained in the introduction, the leading term of inequality \eqref{npoincare} is the one in front of $1/r^2$ for functions supported outside a large ball. Hence, it is particularly important to determine the sharp constant in front of such a term to highlight the non-Euclidean behavior of the inequality. Nevertheless, as happens for inequality \eqref{PR}, the problem of finding the best constant in front of the term $1/r^4$ is still open. See also \cite[Remark 6.1]{BGG}.

 \end{rem}
 
 \begin{rem} 
 It's worth noting that, as happens for inequality \eqref{PR}, the constants appearing in front of the terms $1/r^4$ and $1/\sinh^4 r$ are jointly sharp. In the sense that the inequality
\begin{equation*}
\begin{aligned}
  \int_{\mathbb{H}^{N}}   (\Delta_{\hn} u)^{2}dv_{\hn}  - \left( \frac{N-1}{2} \right)^{2} \int_{\mathbb{H}^{N}} |\nabla_{\hn} u|^2 \ dv_{\hn} \ge a \int_{\hn} \frac{u^2}{r^4} \ dv_{\hn}
+b \int_{\hn} \frac{u^2}{\sinh^4 r} \ dv_{\hn}
\end{aligned}
\end{equation*}
cannot hold for all $ u \in C^{\infty}_{c}(\mathbb{H}^{N})$, or even for for all $ u \in C^{\infty}_{c}(B_\varepsilon)$ given any $\varepsilon>0$, if
\[\begin{aligned}
&a=\frac9{16},\ \ \ b>\frac{(N-1)(N-3)(N^2-4N-3)}{16}\ \ \ \textrm{or}\\
&a>\frac9{16},\ \ \ b=\frac{(N-1)(N-3)(N^2-4N-3)}{16}.
\end{aligned}
\]
This follows by noting that $\sinh r\sim r$ as $r\to0$ and that \eqref{npoincare} yields
\begin{equation*}
\begin{aligned}
\int_{\mathbb{H}^{N}} (\Delta_{\hn} u)^{2} \ dv_{\hn} \geq
 \frac{9}{16} \int_{\mathbb{H}^{N}} \frac{u^2}{r^4} \ dv_{\hn} +\frac{(N-1)(N-3)(N^2-4N-3)}{16} \int_{\hn} \frac{u^2}{\sinh^4 r} \ dv_{\hn}
\end{aligned}
\end{equation*} 
where
\[
\frac9{16}+\frac{(N-1)(N-3)(N^2-4N-3)}{16}=\frac{N^2(N-4)^2}{16}\,
\]
and $\frac{N^2(N-4)^2}{16}$ is the best constant (namely, the larger) for the standard $N$ dimensional Euclidean Rellich inequality, both on the whole ${\mathbb R}^N$ or in any open set containing the origin. 

\end{rem}

Consider the upper half space model for $\hn$, namely $\mathbb{R}^{N}_{+} = \{ (x, y) \in \mathbb{R}^{N-1} \times \mathbb{R}^{+} \} $ endowed with the Riemannian metric $\frac{\delta_{ij}}{y^2}$. We set 
\begin{equation}\label{dd}
d:=d((x,y),(0,1) ): = \cosh^{-1} \left( 1 + \frac{(|y| - 1)^2 + |x|^2}{2 |y|} \right)\,.
\end{equation} 
It is readily seen that $d \sim \log(1/y)$ as $y \rightarrow 0$. By exploiting the transformation
$$
v(x,y) : = y^{\alpha} u(x,y), \quad  x \in \mathbb{R}^{N-1}, y \in \mathbb{R}^{+} ,
$$
with $\alpha=-\frac{N-2}{2}$ or $\alpha=-\frac{N-4}{2}$, from \eqref{npoincare} we derive the following statements 
\begin{cor}\label{cor2}
Let $N >4$ and $d$ as defined in \eqref{dd}. For all  $v \in C_c^{\infty}(\mathbb{R}^{N}_{+})$ the following inequalities hold

\begin{align}\label{HALFrellich1}
& \int_{\mathbb{R}^{+}} \int_{\mathbb{R}^{N-1}} \left( y^2 (\Delta v)^2 + \frac{(N^2 - 2N -1 )}{4} |\nabla v|^2 \right) \ dx \ dy
 \geq \frac{N(N-2)}{16} \int_{\mathbb{R}^{+}} \int_{\mathbb{R}^{N-1}} \frac{v^2}{y^2} \ dx \ dy \notag \\
& + \frac{(N-1)^2}{16} \int_{\mathbb{R}^{+}} \int_{\mathbb{R}^{N-1}}
\frac{v^2}{y^2 d^2} \ dx \ dy + \frac{9}{16} \int_{\mathbb{R}^{+}} \int_{\mathbb{R}^{N-1}} \frac{v^2}{y^2 d^4} \ dx \ dy
\end{align}
\emph{and}

\begin{align}\label{HALFrellich2}
& \int_{\mathbb{R}^{+}} \int_{\mathbb{R}^{N-1}} \left( (\Delta v)^2 + \frac{(N^2 - 2N -9 )}{4} \frac{|\nabla  v|^2}{y^2} \right) \ dx \ dy
 \geq \frac{9}{16} (N+2)(N - 4) \int_{\mathbb{R}^{+}} \int_{\mathbb{R}^{N-1}} \frac{ v^2}{y^4} \ dx \ dy \notag \\
& + \frac{(N-1)^2}{16} \int_{\mathbb{R}^{+}} \int_{\mathbb{R}^{N-1}}
\frac{ v^2}{y^4 d^2} \ dx \ dy + \frac{9}{16} \int_{\mathbb{R}^{+}} \int_{\mathbb{R}^{N-1}} \frac{ v^2}{y^4 d^4} \ dx \ dy.
\end{align}
Furthermore, we have: \par \smallskip\par
$\bullet$ no inequality of the form
$$\int_{\mathbb{R}^{+}} \int_{\mathbb{R}^{N-1}} \left( y^2 (\Delta v)^2 +
c |\nabla v|^2 \right) \ dx \ dy
 \geq \frac{N(N-2)}{16} \int_{\mathbb{R}^{+}} \int_{\mathbb{R}^{N-1}} \frac{v^2}{y^2} \ dx \ dy $$
holds for all $ v \in C^{\infty}_{c}(\mathbb{H}^{N})$ when $c< \frac{( N^2 - 2N -1)}{4}$;
 \par \smallskip\par
$\bullet$ no inequality of the form
$$\int_{\mathbb{R}^{+}} \int_{\mathbb{R}^{N-1}} \left( y^2 (\Delta v)^2 + \frac{(N^2 - 2N  -1)}{4} |\nabla v|^2 \right) \ dx \ dy
 \geq c \int_{\mathbb{R}^{+}} \int_{\mathbb{R}^{N-1}} \frac{v^2}{y^2} \ dx \ dy $$
holds for all $ v \in C^{\infty}_{c}(\mathbb{H}^{N})$ when $c>  \frac{N(N-2)}{16} $;
 \par \smallskip\par
$\bullet$ no inequality of the form
\begin{align*}
& \int_{\mathbb{R}^{+}} \int_{\mathbb{R}^{N-1}} \left( y^2 (\Delta v)^2 + \frac{(N^2 - 2N -1)}{4} |\nabla v|^2 \right) \ dx \ dy
 \geq \frac{N(N-2)}{16} \int_{\mathbb{R}^{+}} \int_{\mathbb{R}^{N-1}} \frac{v^2}{y^2} \ dx \ dy \notag \\
& + c \int_{\mathbb{R}^{+}} \int_{\mathbb{R}^{N-1}}
\frac{v^2}{y^2 d^2} \ dx \ dy
\end{align*}
holds for all $ v \in C^{\infty}_{c}(\mathbb{H}^{N})$ when $c> \frac{(N-1)^2}{16} $.
\par
Similar conclusions hold for the constants $ \frac{(N^2 - 2N - 9)}{2},$ $\frac{9}{16}(N+2)(N-4)$ and  $\frac{(N-1)^2}{16}$
in \eqref{HALFrellich2}.

 \end{cor}

 \par\bigskip\par
 
 \section{Case $k$ arbitrary and $l=0$}\label{l0}
In this section we restate and prove Theorem \ref{main_intro} for $0=l<k$.
\begin{thm}\label{mainhigher0}
Let $k$ be a positive integer and $ N> 2k$. There exist $k$ \emph{positive} constants $c_{k}^i=c_{k}^i(N)$ such that the following inequality holds 
\begin{equation}
\label{global}
\int_{\mathbb{H}^{N}} |\nabla_{\mathbb{H}^{N}}^{k} u|^{2}  \ dv_{\mathbb{H}^{N}} - \left( \frac{N-1}{2} \right)^{2k} 
\int_{\mathbb{H}^{N}} u^2 \ dv_{\mathbb{H}^{N}}\ge \sum_{i= 1}^{k} c_{k}^i\int_{\mathbb{H}^{N}} \frac{u^2}{r^{2i}} \ dv_{\mathbb{H}^{N}}\,,
\end{equation}
for all $u \in C^{\infty}_{0}(\mathbb{H}^{N})$. Furthermore, the leading terms a $r\rightarrow 0$ and $r \rightarrow\infty$ are explicitly given by
$$ c_{k}^1=d_k:=
\left \{\begin{array}{ll} 
\sum_{j = 1}^{ m} \frac{(N-1)^{4m-2j}}{2^{4m-1}}& \text{if $k=2m$}\,,\\
\sum_{j = 1}^{ m} \frac{(N-1)^{4m-2j+2}}{2^{4m+1}}+ \frac{(N-1)^{2m}}{2^{4m+2}} & \text{if $k=2m+1$\,,}
 \end{array}\right.$$
 $$c_{k}^k=e_k:=
\left \{\begin{array}{ll} 
\frac{9}{2^{4m}} \prod_{j = 1}^{ m -1} (N+ 4j)^2(N-4j-4)^2 & \text{if $k=2m$}\,,\\
\frac{1}{2^{4m+2}} \prod_{j = 1}^{ m} (N+ 4j - 2)^2(N- 4j - 2)^2 & \text{if $k=2m+1$\,,}
 \end{array}\right.$$
where we use the conventions: $\sum_{j = 1}^{ 0}=0$ and $\prod_{j = 1}^{ 0}=1$.

\end{thm}

\begin{proof}
Here and after, for shortness we will write $\Delta_{\mathbb{H}^{N}}=\Delta$. In the proof we will repeatedly exploit the following inequality from \cite[Theorem 4.4]{Yang}:
\begin{align}\label{yangsukong}
\int_{\mathbb{H}^{N}} \frac{(\Delta u)^2}{r^{\beta}} \ dv_{\mathbb{H}^{N}}\geq &\frac{(N+\beta)^2(N-\beta-4)^2}{16}  \int_{\mathbb{H}^{N}} \frac{u^2}{r^{4+\beta}} \ dv_{\mathbb{H}^{N}} \notag  \\
& +\frac{(N-2-\beta)(N-2+\beta)(N-1)}{8}  \int_{\mathbb{H}^{N}} \frac{u^2}{r^{2+\beta}} \ dv_{\mathbb{H}^{N}}+\frac{(N-1)^2}{16}  \int_{\mathbb{H}^{N}} \frac{u^2}{r^{\beta}} \ dv_{\mathbb{H}^{N}} , 
\end{align}
for all $u \in C^{\infty}_{0}(\mathbb{H}^{N})$ and $0\leq \beta<N-4$.

We prove separately the case $k$ even and $k$ odd. First we assume $k=2m$ even. If $m = 1$, \eqref{global} follows directly from \eqref{PR}. When $m = 2$, by \eqref{PR} and \eqref{yangsukong} with $N>8$, we have
\begin{align*}
\int_{\mathbb{H}^{N}} (\Delta (\Delta u))^2 \ dv_{\mathbb{H}^{N}} & \geq \frac{(N-1)^4}{16} \int_{\mathbb{H}^{N}} (\Delta u)^2 \ dv_{\mathbb{H}^{N}} 
+ \frac{(N-1)^2}{8} \int_{\mathbb{H}^{N}} \frac{(\Delta u)^2}{r^2} \ dv_{\mathbb{H}^{N}} + \frac{9}{16} \int_{\mathbb{H}^{N}} \frac{(\Delta u)^2}{r^4} \ dv_{\mathbb{H}^{N}}  \notag\\
& \geq \frac{(N-1)^4}{16} \left[ \frac{(N-1)^4}{16} \int_{\mathbb{H}^{N}} u^2 \ dv_{\mathbb{H}^{N}} + \frac{(N-1)^2}{8} \int_{\mathbb{H}^{N}}  \frac{u^2}{r^2} \ dv_{\mathbb{H}^{N}}
+ \frac{9}{16} \int_{\mathbb{H}^{N}} \frac{u^2}{r^4} \ dv_{\mathbb{H}^{N}} \right] \notag \\
& + \frac{(N-1)^2}{8} \left[ \frac{(N + 2)^2 (N - 6)^2}{16} \int_{\mathbb{H}^{N}} \frac{ u^2}{r^6} \ dv_{\mathbb{H}^{N}} 
+  \frac{N(N-1)(N-4)}{8} \int_{\mathbb{H}^{N}}  \frac{u^2}{r^4} \ dv_{\mathbb{H}^{N}} \right.  \notag \\
& \left. + \frac{(N-1)^2}{16} \int_{\mathbb{H}^{N}}  \frac{ u^2}{r^2} \ dv_{\mathbb{H}^{N}}   \right] + \frac{9}{16} \left[ \frac{(N + 4)^2(N - 8)^2}{16} \int_{\mathbb{H}^{N}} 
\frac{u^2}{r^8} \ dv_{\mathbb{H}^{N}} \right. \notag \\
& \left.  + \frac{(N + 2)(N - 1) (N - 6)}{8} \int_{\mathbb{H}^{N}}  \frac{u^2}{r^6} \ dv_{\mathbb{H}^{N}}  + \frac{(N - 1)^2}{16} \int_{\mathbb{H}^{N}} 
\frac{u^2}{r^4} \ dv_{\mathbb{H}^{N}} \right] \notag \\
& = \left( \frac{N-1}{2} \right)^8 \int_{\mathbb{H}^{N}} u^2 \ dv_{\mathbb{H}^{N}} + \sum_{i = 1}^{4} c_{4}^i \int_{\mathbb{H}^{N}} \frac{u^2}{r^{2i}},
\end{align*}
where $c_{4}^1 = \frac{(N - 1)^6}{2^7} + \frac{(N - 1)^4}{2^7},$ $c_{4}^4 = \frac{9}{2^8} {(N + 4)^2 (N - 8)^2}$. Hence, \eqref{global} is proved for $k= 4.$ \\
\par
Next we assume \eqref{global} holds for $k=2m$ with $m > 2,$ namely
\begin{align}\label{higher1}
\int_{\mathbb{H}^{N}} (\Delta^{m} u)^2 \ dv_{\mathbb{H}^{N}} - \left( \frac{N-1}{2} \right)^{4m} \int_{\mathbb{H}^{N}} u^{2} \ dv_{\mathbb{H}^{N}} & \geq \sum_{j = 1}^{ m} \frac{(N-1)^{4m-2j}}{2^{4m-1}} \int_{\mathbb{H}^{N}} \frac{u^2}{r^{2}} \ dv_{\mathbb{H}^{N}}+\sum_{i = 2}^{2m - 1} c_{2m}^i \int_{\mathbb{H}^{N}} \frac{u^2}{r^{2i}} \ dv_{\mathbb{H}^{N}}  \notag  \\
&+\frac{9}{2^{4m}} \prod_{j = 1}^{ m -1} (N+ 4j)^2(N-4(j+1))^2 \int_{\mathbb{H}^{N}} \frac{u^2}{r^{4m}} \ dv_{\mathbb{H}^{N}}, 
\end{align}
where, for $2\leq i\leq 2m-1$, the $c_{2m}^i$ are suitable positive constants and $N>4m$. Inequality \eqref{higher1} yields
\begin{align*}
\int_{\mathbb{H}^{N}} (\Delta^{m+1} u)^2 = \int_{\mathbb{H}^{N}} (\Delta^{m} (\Delta u))^2\geq 
\left( \frac{N-1}{2} \right)^{4m} \int_{\mathbb{H}^{N}}  (\Delta u)^2 \ dv_{\mathbb{H}^{N}} +\sum_{j = 1}^{ m} \frac{(N-1)^{4m-2j}}{2^{4m-1}} \int_{\mathbb{H}^{N}} \frac{(\Delta u)^2}{r^{2}} \ dv_{\mathbb{H}^{N}}\notag \\
+ \sum_{ i = 2}^{2m - 1} c_{2m}^i \int_{\mathbb{H}^{N}} \frac{(\Delta u)^2}{r^{2i}} \ dv_{\mathbb{H}^{N}}+ \frac{9}{2^{4m}} \prod_{i = 1}^{m -1} (N + 4j)^2(N-4(j + 1))^2 
\int_{\mathbb{H}^{N}} \frac{(\Delta u)^2}{r^{4m}} \ dv_{\mathbb{H}^{N}} \,.
\end{align*}
Next, by \eqref{yangsukong}, for  $N > 4m + 4$ we have
\begin{align*}\int_{\mathbb{H}^{N}}  \frac{(\Delta u)^2}{r^{2}} \ dv_{\mathbb{H}^{N}} &\geq\frac{(N+2)^2(N-6)^2}{16}  \int_{\mathbb{H}^{N}} \frac{u^2}{r^{6}} \ dv_{\mathbb{H}^{N}}  \notag  \\ &+\frac{(N-4)(N)(N-1)}{8}  \int_{\mathbb{H}^{N}} \frac{u^2}{r^{4}} \ dv_{\mathbb{H}^{N}}+\frac{(N-1)^2}{16}  \int_{\mathbb{H}^{N}} \frac{u^2}{r^{2}} \ dv_{\mathbb{H}^{N}}\,,
\end{align*}
\begin{align*} \sum_{ i = 2}^{2m - 1} c_{2m}^i \int_{\mathbb{H}^{N}} \frac{(\Delta u)^2}{r^{2i}} \ dv_{\mathbb{H}^{N}}\geq  \sum_{ i = 2}^{2m+1} \bar c_{2m}^i \int_{\mathbb{H}^{N}} \frac{u^2}{r^{2i}} \ dv_{\mathbb{H}^{N}}\,,
\end{align*}
\begin{align*}
\int_{\mathbb{H}^{N}} \frac{(\Delta u)^2}{r^{4m}} \ dv_{\mathbb{H}^{N}}&\geq \frac{(N+4m)^2(N-4m-4)^2}{16}  \int_{\mathbb{H}^{N}} \frac{u^2}{r^{4+4m}} \ dv_{\mathbb{H}^{N}}\notag  \\ &+\frac{(N-2-4m)(N-2+4m)(N-1)}{8}  \int_{\mathbb{H}^{N}} \frac{u^2}{r^{2+4m}} \ dv_{\mathbb{H}^{N}}+\frac{(N-1)^2}{16}  \int_{\mathbb{H}^{N}} \frac{u^2}{r^{4m}} \ dv_{\mathbb{H}^{N}}\,,
\end{align*}
where, for $2\leq i\leq 2m+1$, $\bar c_{2m}^i$ are suitable positive constants.
The above inequalities and \eqref{PR}, finally yield
\begin{align*}
\int_{\mathbb{H}^{N}} (\Delta^{m+1} u)^2  - \left( \frac{N-1}{2} \right)^{4(m + 1)} \int_{\mathbb{H}^{N}} u^2 \ dv_{\mathbb{H}^{N}}   \geq 
\left(\frac{(N-1)^{4m+2}}{2^{4m+3}}+\sum_{j = 1}^{ m} \frac{(N-1)^{4m-2j+2}}{2^{4m+3}}   \right)\int_{\mathbb{H}^{N}} \frac{u^2}{r^{2}} \ dv_{\mathbb{H}^{N}}
\notag\\ + \sum_{ i = 2}^{2m+1} \hat c_{2m}^i \int_{\mathbb{H}^{N}} \frac{u^2}{r^{2i}} \ dv_{\mathbb{H}^{N}}+
\notag\\
 + \frac{9}{2^{4m}} \prod_{j = 1}^{m-1} (N + 4j)^2(N-4(j + 1))^2  
\frac{(N+4m)^2(N-4m-4)^2}{16}  \int_{\mathbb{H}^{N}} \frac{u^2}{r^{4+4m}} \ dv_{\mathbb{H}^{N}}\\
=\sum_{j = 1}^{ m+1} \frac{(N-1)^{4(m+1)-2j}}{2^{4(m+1)-1}} \int_{\mathbb{H}^{N}} \frac{u^2}{r^{2}} \ dv_{\mathbb{H}^{N}}+ \sum_{ l = 2}^{2(m+1)-1} \hat c_{2m}^i \int_{\mathbb{H}^{N}} \frac{u^2}{r^{2i}} \ dv_{\mathbb{H}^{N}}
\notag\\
 + \frac{9}{2^{4m+4}} \prod_{j = 1}^{m } (N + 4j)^2(N-4(j + 1))^2    \int_{\mathbb{H}^{N}} \frac{u^2}{r^{4+4m}} \ dv_{\mathbb{H}^{N}}
\end{align*}
where, for $2\leq i\leq 2m+1$, $\hat c_{2m}^i$ are suitable positive constants. By induction, this completes the proof of \eqref{global} for $k$ even.\\

Next we turn to the case $k=2m+1$ odd.  For $m = 1$ and $N > 6$, by  \eqref{poincareeq},  \eqref{PR} and \eqref{yangsukong}, we deduce
\begin{align*}
\int_{\mathbb{H}^{N}} |\nabla (\Delta u)|^2 \ dv_{\mathbb{H}^{N}}  & \geq \left( \frac{N-1}{2} \right)^2 \int_{\mathbb{H}^{N}} (\Delta u)^2 \ dv_{\mathbb{H}^{N}}
 + \frac{1}{4} \int_{\mathbb{H}^{N}} \frac{(\Delta u)^2}{r^2} \ dv_{\mathbb{H}^{N}} \notag \\
& \geq  \left( \frac{N-1}{2} \right)^2 \left[ \left( \frac{N-1}{2} \right)^2\int_{\mathbb{H}^{N}} u^2 \ dv_{\mathbb{H}^{N}} + \frac{(N-1)^2}{8} \int_{\mathbb{H}^{N}} \frac{u^2}{r^2} \ dv_{\mathbb{H}^{N}} 
+ \frac{9}{16} \int_{\mathbb{H}^{N}} \frac{u^2}{r^4} \ dv_{\mathbb{H}^{N}} \right] \notag \\
& +  \frac{1}{4}  \left[ \frac{(N + 2)^2 (N-6)^2}{16} \int_{\mathbb{H}^{N}} \frac{u^2}{r^6} \ dv_{\mathbb{H}^{N}} + \frac{N(N-1)(N-4)}{8} \int_{\mathbb{H}^{N}} \frac{u^2}{r^4} \ dv_{\mathbb{H}^{N}}    \right. \notag \\
& \left. + \frac{(N-1)^2}{16} \int_{\mathbb{H}^{N}} \frac{u^2}{r^2} \ dv_{\mathbb{H}^{N}} \right]. 
\end{align*}
 Hence we obtain,
\[
\int_{\mathbb{H}^{N}} |\nabla (\Delta u)|^2 \ dv_{\mathbb{H}^{N}} - \left( \frac{N-1}{2} \right)^6 \int_{\mathbb{H}^{N}} u^2 \ dv_{\mathbb{H}^{N}} \geq   \sum_{l = 1}^{3} c_{3}^i \int_{\mathbb{H}^{N}} \frac{u^2}{r^{2i}},
\]
where $c_{3}^1 =  \left( \frac{(N-1)^4}{2^5} + \frac{(N-1)^2}{2^6} \right)$ and $c_{3}^3=\frac{1}{4^3} (N+2)^2(N-6)^2$ hence, \eqref{global} with $m=1$ is verified. For general $ m + 1$ the proof follows very similar to the case $k$ even, we skip the details for brevity. This completes the proof.  
\end{proof}

\section{Case $k>l>0$ arbitrary} \label{karbitrary}
In this section we restate and prove Theorem \ref{main_intro} for $l>0$, the case $l=0$ has already been dealt with in Section \ref{l0}. 

 \begin{thm}\label{mainmain}
Let $k>l$ be positive integers and $ N> 2k$. There exist $k$ \emph{positive} constants $\alpha_{k,l}^i=\alpha_{k,l}^i(N)$ such that the following inequality holds 
$$
\int_{\mathbb{H}^{N}} |\nabla_{\mathbb{H}^{N}}^{k} u|^{2}  \ dv_{\mathbb{H}^{N}} -\left( \frac{N-1}{2} \right)^{2(k - l)} 
\int_{\mathbb{H}^{N}} |\nabla_{\mathbb{H}^{N}}^{l} u|^2 \ dv_{\mathbb{H}^{N}}\ge \sum_{i = 1}^{k} \alpha_{k,l}^i \int_{\mathbb{H}^{N}} \frac{u^2}{r^{2i}} \ dv_{\mathbb{H}^{N}}\,,
$$
for all $u \in C^{\infty}_{0}(\mathbb{H}^{N})$. Furthermore, the leading terms a $r\rightarrow 0$ and $r \rightarrow\infty$, namely $\alpha_{k,l}^1 $ and $\alpha_{k,l}^k$ are explicitly given as follows
$$\alpha_{k,l}^1:=
\left \{\begin{array}{lllll} 
d_{2(m-h)} \,a_{h} & \text{if $k=2m$ and $l=2h$}\,,\\
2^{4(m-h)-2}\,a_{2(m-h)}\,a_{h}+d_{2(m-h-1)} \,a_{h+1} & \text{if $k=2m$ and $l=2h+1$, $h\neq m-1$}\,,\\
 a_1\,a_{m-1}   & \text{if $k=2m$ and $l=2m-1$}\,,\\
4a_1\, a_{h} d_{2(m-h)} + \frac{1}{4} a_{m} & \text{if $k=2m+1$ and $l=2h$}\,,\\
 \frac{1}{4} a_{m} + 4^{2(m-h)}\, a_{2(m-h)}\,a_{h}\,a_1+4\,d_{2(m - h -1)} a_{h + 1} a_1  & \text{if $k=2m+1$ and $l=2h+1$, $h\neq m-1$}\,,\\
 \frac{1}{4} a_{m} + 4\,a_2\, a_{m - 1} & \text{if $k=2m+1$ and $l=2m-1$}\,,
 \end{array}\right.
 $$

  $$
 \alpha_{k,l}^k:=
\left \{\begin{array}{lllll} 
e_{2(m-h)}\, b_{h,4(m-h)}& \text{if $k=2m$ and $l=2h$}\,,\\
e_{2(m-h-1)}\, b_{h+1,4(m-h-1)} & \text{if $k=2m$ and $l=2h+1$, $h\neq m-1$}\,,\\
\frac{9}{16}\, b_{m-1,4}  & \text{if $k=2m$ and $l=2m-1$}\,,\\
\frac{1}{4} b_{m,2} & \text{if $k=2m+1$ and $l<k$}\,,

 \end{array}\right.$$

where $a_0=1$, $d_0=0$ and, for any $\gamma$ and $\beta$ positive integers, $a_{\gamma}=\frac{(N-1)^{2\gamma}}{2^{4\gamma}}$, $b_{\gamma,\beta} = \prod_{j = 0}^{\gamma-1}\frac{(N+(\beta+4j))^2(N-(\beta+4j)-4)^2}{16}$, $d_{\gamma}$ and $e_{\gamma}$ are the constants defined in Theorem \ref{mainhigher0}.
\end{thm}

The proof is achieved by considering separately four cases. In each proof we will exploit the following technical lemma whose proof can be obtained by induction, iterating \eqref{yangsukong}. Notice that, except for the main statements, for shortness we will always write $\Delta_{\mathbb{H}^{N}}=\Delta$.

\begin{lem}
Let $\gamma$ be a positive integer. For all $u \in C^{\infty}_{0}(\mathbb{H}^{N})$ and $0\leq \beta<N-4\gamma$, the following inequality holds
\begin{align}\label{yangxtended}
\int_{\mathbb{H}^{N}} \frac{(\Delta^{\gamma} u)^2}{r^{\beta}} \ dv_{\mathbb{H}^{N}}\geq a_{\gamma}\int_{\mathbb{H}^{N}} \frac{u^2}{r^{\beta}} \ dv_{\mathbb{H}^{N}} +\sum_{ j= 1}^{2\gamma-1} a_{\gamma,\beta}^j  \int_{\mathbb{H}^{N}} \frac{u^2}{r^{2j+\beta}} \ dv_{\mathbb{H}^{N}}+b_{\gamma,\beta} \int_{\mathbb{H}^{N}} \frac{u^2}{r^{4\gamma+\beta}} \ dv_{\mathbb{H}^{N}}   \,, 
\end{align}
where $a_{\gamma}=\frac{(N-1)^{2\gamma}}{2^{4\gamma}}$, $a_{m,\beta}^j$ are suitable positive constants and $b_{\gamma,\beta} = \prod_{j = 0}^{\gamma-1}\frac{(N+(\beta+4j))^2(N-(\beta+4j)-4)^2}{16}$.

\end{lem}

\par \bigskip\par
 \subsection{Case $k=2m$ even and $l=2h$ even}
  \begin{thm}\label{mainhigheree}
Let $m,h$ be integers such that $0<h< m$ and $ N> 4m$. There exist $2m$ \emph{positive} constants $\alpha^i=\alpha^i(N,m,h)$ such that the following inequality holds  
 \begin{align}\label{highereven1}
\int_{\mathbb{H}^{N}}  (\Delta_{\hn}^{m} u)^2 \ dv_{\mathbb{H}^{N}} - \left( \frac{N-1}{2} \right)^{4(m-h)} \int_{\mathbb{H}^{N}}  (\Delta_{\hn}^{h} u)^{2} \ dv_{\mathbb{H}^{N}}  \geq 
  \sum_{i = 1}^{2m} \alpha^{i} \int_{\mathbb{H}^{N}} \frac{u^2}{r^{2i}} \ dv_{\mathbb{H}^{N}},
\end{align}
for all $u \in C^{\infty}_{0}(\mathbb{H}^{N})$. Furthermore, the leading terms a $r\rightarrow 0$ and $r \rightarrow\infty$ are explicitly given by
$$ \alpha^1:=d_{2(m-h)}\,a_h \quad \text{and}\quad \alpha^{2m}:=e_{2(m-h)}\, b_{h,4(m-h)} \,,$$
where, for any $\gamma$ and $\beta$ positive integers, $a_{\gamma}=\frac{(N-1)^{2\gamma}}{2^{4\gamma}}$, $b_{\gamma,\beta} = \prod_{j = 0}^{\gamma-1}\frac{(N+(\beta+4j))^2(N-(\beta+4j)-4)^2}{16}$, $d_{\gamma}$ and $e_{\gamma}$ are the constants defined in Theorem \ref{mainhigher0}. 
\end{thm}

 \begin{proof}
 By applying \eqref{global} with $k=2(m-h)$ and \eqref{yangxtended} with $\gamma=h$ and $\beta=2i$ we deduce
 $$\int_{\mathbb{H}^{N}}  (\Delta^{m} u)^2 \ dv_{\mathbb{H}^{N}} =\int_{\mathbb{H}^{N}}  (\Delta^{m-h}(\Delta^h u))^2 \ dv_{\mathbb{H}^{N}} $$
 $$\geq \left( \frac{N-1}{2} \right)^{4(m-h)} 
\int_{\mathbb{H}^{N}} (\Delta^h u)^2 \ dv_{\mathbb{H}^{N}}+ \sum_{i= 1}^{2(m-h)} c_{2(m-h)}^i\int_{\mathbb{H}^{N}} \frac{(\Delta^h u)^2}{r^{2i}} \ dv_{\mathbb{H}^{N}}\geq \left( \frac{N-1}{2} \right)^{4(m-h)} 
\int_{\mathbb{H}^{N}} (\Delta^h u)^2 \ dv_{\mathbb{H}^{N}}$$
$$+ \sum_{i= 1}^{2(m-h)} c_{2(m-h)}^i  \left(a_{h}\int_{\mathbb{H}^{N}} \frac{u^2}{r^{2i}} \ dv_{\mathbb{H}^{N}} +\sum_{ j= 1}^{2h-1} a_{h,2i}^j  \int_{\mathbb{H}^{N}} \frac{u^2}{r^{2j+2i}} \ dv_{\mathbb{H}^{N}}+b_{h,2i} \int_{\mathbb{H}^{N}} \frac{u^2}{r^{4h+2i}} \ dv_{\mathbb{H}^{N}} \right)\,,$$
where all the constants are positive. Settled $g(j,i):=2j+2i$, for $0\leq j\leq 2h$ and $1\leq i\leq 2(m-h)$ it is readily verified that $g$ has a unique global minimum $g(0,1)=1$ and a unique global maximum $g(2h,2(m-h))=4m$. Furthermore, by the fact that $g(j,1)$ goes monotonically from $2$ to $4h+2$ and $g(2h,i)$ goes monotonically from $4h+2$ to $4m$, we deduce the existence of $2m$ positive constants $\alpha^i=\alpha^i(N,m,h)$ such that \eqref{highereven1} holds. Moreover, $$ \alpha^1=d_{2(m-h)} \,a_{h} \quad \text{and} \quad  \alpha^k:=e_{2(m-h)}\, b_{h,4(m-h)}\,.$$
 \end{proof}
 
 \par \medskip\par

 \subsection{Case $k=2m$ even and $l=2h+1$ odd}

 \begin{thm}\label{mainhighereo}
Let $m,h$ be integers such that $0\leq h< m$ and $ N> 4m$. There exist $2m$ \emph{positive} constants $\bar\alpha^i=\bar \alpha^i(N,m,h)$ such that the following inequality holds  
 \begin{align}\label{highereven2}
\int_{\mathbb{H}^{N}}  (\Delta_{\hn}^{m} u)^2 \ dv_{\mathbb{H}^{N}} - \left( \frac{N-1}{2} \right)^{4(m-h)-2} \int_{\mathbb{H}^{N}}  |\nabla_{\hn} \Delta_{\hn}^{h} u|^{2} \ dv_{\mathbb{H}^{N}}  \geq 
  \sum_{i = 1}^{2m}\bar \alpha^{i} \int_{\mathbb{H}^{N}} \frac{u^2}{r^{2i}} \ dv_{\mathbb{H}^{N}}\,,
\end{align}
for all $u \in C^{\infty}_{0}(\mathbb{H}^{N})$. Furthermore, the leading terms a $r\rightarrow 0$ and $r \rightarrow\infty$ are explicitly given by:\par\medskip\par
\noindent
if $0\leq h<m-1$
$$\bar \alpha^1:=2^{4(m-h)-2}\,a_{2(m-h)}\,a_{h}+d_{2(m-h-1)} \,a_{h+1} \quad \text{and} \quad \bar\alpha^k:=e_{2(m-h-1)}\, b_{h+1,4(m-h-1)}\,,$$
if $h=m-1$ 
$$\bar  \alpha^1:=a_1\,a_{m-1} \quad \text{and} \quad \bar\alpha^k:=\frac{9}{16}\, b_{m-1,4}\,.$$
where $a_0=1$ and, for any $\gamma$ and $\beta$ positive integers, $a_{\gamma}=\frac{(N-1)^{2\gamma}}{2^{4\gamma}}$, $b_{\gamma,\beta} = \prod_{j = 0}^{\gamma-1}\frac{(N+(\beta+4j))^2(N-(\beta+4j)-4)^2}{16}$, $d_{\gamma}$ and $e_{\gamma}$ are the constants defined in Theorem \ref{mainhigher0}.
\end{thm}

 \begin{proof}
 Let $0< h<m-1$, by applying first \eqref{global} with $k=2(m-h-1)$, then \eqref{npoincare} and finally \eqref{yangxtended} with $\gamma=h,h+1$ and $\beta=2,4,2i$, we deduce
 $$
 \int_{\mathbb{H}^{N}}  (\Delta^{m} u)^2 \ dv_{\mathbb{H}^{N}} =\int_{\mathbb{H}^{N}}  (\Delta^{m-h-1}(\Delta^{h+1} u))^2 \ dv_{\mathbb{H}^{N}} $$
$$ \geq \left( \frac{N-1}{2} \right)^{4(m-h-1)} \int_{\mathbb{H}^{N}} (\Delta^{h+1} u)^2 \ dv_{\mathbb{H}^{N}}+ \sum_{i= 1}^{2(m-h-1)} c_{2(m-h-1)}^i\int_{\mathbb{H}^{N}} \frac{(\Delta^{h+1} u)^2}{r^{2i}} \ dv_{\mathbb{H}^{N}} $$
$$ \geq \left( \frac{N-1}{2} \right)^{4(m-h)-2} \int_{\mathbb{H}^{N}} |\nabla \Delta^{h} u|^2 \ dv_{\mathbb{H}^{N}}
 +\frac{1}{4}\left( \frac{N-1}{2} \right)^{4(m-h)-2} \int_{\mathbb{H}^{N}} \frac{(\Delta^{h} u)^2}{r^{2}} \ dv_{\mathbb{H}^{N}}$$
 $$+ \frac{9}{16} \left( \frac{N-1}{2} \right)^{4(m-h-1)} \int_{\mathbb{H}^{N}} \frac{(\Delta^{h} u)^2}{r^{4}} \ dv_{\mathbb{H}^{N}} 
  + \sum_{i= 1}^{2(m-h-1)} c_{2(m-h-1)}^i\int_{\mathbb{H}^{N}} \frac{(\Delta^{h+1} u)^2}{r^{2i}} \ dv_{\mathbb{H}^{N}}$$
  $$
\geq \left( \frac{N-1}{2} \right)^{4(m-h)-2} \int_{\mathbb{H}^{N}} |\nabla \Delta^{h} u|^2 \ dv_{\mathbb{H}^{N}}$$
$$ +\frac{1}{4}\left( \frac{N-1}{2} \right)^{4(m-h)-2} \left( a_{h}\int_{\mathbb{H}^{N}} \frac{u^2}{r^{2}} \ dv_{\mathbb{H}^{N}} +\sum_{ j= 1}^{2h-1} a_{h,2}^j  \int_{\mathbb{H}^{N}} \frac{u^2}{r^{2j+2}} \ dv_{\mathbb{H}^{N}}+b_{h,2} \int_{\mathbb{H}^{N}} \frac{u^2}{r^{4h+2}} \ dv_{\mathbb{H}^{N}}   \right)$$
$$+\frac{9}{16} \left( \frac{N-1}{2} \right)^{4(m-h-1)}  \left(   a_{h}\int_{\mathbb{H}^{N}} \frac{u^2}{r^{4}} \ dv_{\mathbb{H}^{N}} +\sum_{ j= 1}^{2h-1} a_{h,4}^j  \int_{\mathbb{H}^{N}} \frac{u^2}{r^{2j+4}} \ dv_{\mathbb{H}^{N}}+b_{h,4} \int_{\mathbb{H}^{N}} \frac{u^2}{r^{4h+4}} \ dv_{\mathbb{H}^{N}}  \right)
$$
$$+\sum_{i= 1}^{2(m-h-1)} c_{2(m-h-1)}^i \left(a_{h+1}\int_{\mathbb{H}^{N}} \frac{u^2}{r^{2i}} \ dv_{\mathbb{H}^{N}} +\sum_{ j= 1}^{2h+1} a_{h+1,2i}^j  \int_{\mathbb{H}^{N}} \frac{u^2}{r^{2j+2i}} \ dv_{\mathbb{H}^{N}}+b_{h+1,2i} \int_{\mathbb{H}^{N}} \frac{u^2}{r^{4h+4+2i}} \ dv_{\mathbb{H}^{N}} \right)\,.
$$

Hence, with an argument similar to that applied in the proof of Theorem \ref{mainhigheree}, it's readily deduced the existence of $2m$ positive constants $\bar\alpha^i=\alpha^i(N,m,h)$ such that \eqref{highereven2} holds. Furthermore,
$$\bar \alpha^1=\frac{1}{4}\left( \frac{N-1}{2} \right)^{4(m-h)}\,a_{h}+d_{2(m-h-1)} \,a_{h+1} \quad \text{and} \quad \bar\alpha^k=e_{2(m-h-1)}\, b_{h+1,4(m-h-1)}\,.$$
When $h=0$ the above computations may be slightly modified to show the validity of \eqref{highereven2}. Furthermore, by setting $a_0=0$, the leading terms are still given as above.\par
When $h=m-1$, the existence of $2m$ positive constants $\bar\alpha^i=\alpha^i(N,m,h)$ such that \eqref{highereven2} holds follows similarly by applying first \eqref{npoincare} and then \eqref{yangxtended}, with $\gamma=m-1$ and $\beta=2,4$. Finally, the leading terms turn out to be
$$ \bar\alpha^1= \frac{(N-1)^2}{16}\,a_{m-1} \quad \text{and} \quad \bar \alpha^k=\frac{9}{16}\, b_{m-1,4}\,.$$

\end{proof}

 \par \medskip\par
 \subsection{Case $k=2m + 1$ odd and $l=2h$ even}

 \begin{thm}
Let $m,h$ be integers such that $0< h\leq m$ and $ N> 4m + 2$. There exist $2m+1$ \emph{positive} constants $\delta^i=\delta^i(N,m,h)$ such that the following inequality holds
 
 \begin{align}\label{higherodd1}
\int_{\mathbb{H}^{N}}  |\nabla_{\hn} (\Delta_{\hn}^{m} u) |^2 \ dv_{\mathbb{H}^{N}} - \left( \frac{N-1}{2} \right)^{4(m-h) + 2} \int_{\mathbb{H}^{N}}  ( \Delta_{\hn}^{h} u)^{2} \ dv_{\mathbb{H}^{N}}  \geq 
  \sum_{i = 1}^{2m + 1} \delta^{i} \int_{\mathbb{H}^{N}} \frac{u^2}{r^{2i}} \ dv_{\mathbb{H}^{N}},
\end{align}
for all $u \in C^{\infty}_{0}(\mathbb{H}^{N})$. Furthermore, the leading terms as $r\rightarrow 0$ and $r \rightarrow\infty$ are explicitly given by:\par\medskip\par
$$\delta^1:= 4a_1\, a_{h} d_{2(m-h)} + \frac{1}{4} a_{m}\quad \text{and} \quad   \delta^{2m +1} := \frac{1}{4} b_{m,2} $$

where $d_0=0$ and, for any $\gamma$ and $\beta$ positive integers, $a_{\gamma}=\frac{(N-1)^{2\gamma}}{2^{4\gamma}}$, $b_{\gamma,\beta} = \prod_{j = 0}^{\gamma-1}\frac{(N+(\beta+4j))^2(N-(\beta+4j)-4)^2}{16}$, $d_{\gamma}$ and $e_{\gamma}$ are the constants defined in Theorem \ref{mainhigher0}.

\end{thm}
 
 \begin{proof}
From \eqref{poincareeq} we know
 
 \begin{align*}
 \int_{\hn} |\nabla (\Delta^m u)|^2 \ dv_{\hn} & \geq \left( \frac{N-1}{2} \right)^{2} \int_{\hn} (\Delta^{m} u)^2 \ dv_{\hn} + \frac{1}{4} \int_{\hn} \frac{(\Delta^m u)^2}{r^2} \ dv_{\hn} \,.
 \end{align*} 
If $0< h< m$, from \eqref{highereven1} and \eqref{yangxtended} we readily get
 $$\int_{\hn} |\nabla (\Delta^m u)|^2 \ dv_{\hn} \geq  \left( \frac{N-1}{2} \right)^{2} \left(\left( \frac{N-1}{2} \right)^{4(m-h)} \int_{\mathbb{H}^{N}}  (\Delta_{\hn}^{h} u)^{2} \ dv_{\mathbb{H}^{N}}  +
  \sum_{i = 1}^{2m} \alpha^{i} \int_{\mathbb{H}^{N}} \frac{u^2}{r^{2i}} \ dv_{\mathbb{H}^{N}} \right)$$
  $$+ \frac{1}{4} \left(   a_{m} \int_{\hn} \frac{u^2}{r^2} \ dv_{\hn} + \sum_{j = 1}^{2m-1} a^{j}_{m,2} \int_{\hn} \frac{u^2}{r^{2j + 2}} \ dv_{\hn} 
 + b_{m,2} \int_{\hn} \frac{u^2}{r^{4m + 2}} \right)$$
by which the existence of $2m+1$ positive constants $\delta^i=\delta^i(N,m,h)$ such that \eqref{higherodd1} holds follows. Furthermore, one has
$$ \delta^1:= \left( \frac{N-1}{2} \right)^{2} a_{h} d_{2(m-h)} + \frac{1}{4} a_{m} \quad \text{and} \quad  \delta^{2m +1} := \frac{1}{4} b_{m,2}\,.$$
When $h=m$ the same proof may be adopted without applying \eqref{highereven1}. In this case, the leading terms are defined as above by assuming $d_0=0$. 
\end{proof}
 
 \par \medskip\par
 \subsection{Case $k=2m + 1$ odd and $l=2h + 1$ odd}

 \begin{thm}
Let $m,h$ be integers such that $0\leq h< m$ and $ N> 4m + 2$. There exist $2m+1$ \emph{positive} constants $\bar \delta^i=\bar\delta^i(N,m,h)$ such that the following inequality holds 
 \begin{align}\label{higherodd2}
\int_{\mathbb{H}^{N}}  |\nabla_{\hn} (\Delta_{\hn}^{m} u) |^2 \ dv_{\mathbb{H}^{N}} - \left( \frac{N-1}{2} \right)^{4(m-h)} \int_{\mathbb{H}^{N}}  | \nabla_{\hn} (\Delta_{\hn}^{h} u)|^{2} \ dv_{\mathbb{H}^{N}}  \geq 
  \sum_{i = 1}^{2m + 1}\bar \delta^{i} \int_{\mathbb{H}^{N}} \frac{u^2}{r^{2i}} \ dv_{\mathbb{H}^{N}},
\end{align}
for all $u \in C^{\infty}_{0}(\mathbb{H}^{N})$. Furthermore, the leading terms as $r\rightarrow 0$ and $r \rightarrow\infty$ are explicitly given by:\par\medskip\par

 \noindent
 if $0\leq h<m-1$
$$ \bar\delta^1:= \frac{1}{4} a_{m} + 4^{2(m-h)}\, a_{2(m-h)}\,a_{h}\,a_1+4\,d_{2(m - h -1)} a_{h + 1} a_1  
\mbox{ and } \quad \bar\delta^{2m + 1} := \frac{1}{4} b_{m, 2}, $$
if $h=m-1$ 
$$ \bar\delta^1:= \frac{1}{4} a_{m} + 4\,a_2\, a_{m - 1}  \quad \text{ and } \quad \bar \delta^{2m+1} = \frac{1}{4} b_{m,2},$$
where $a_0=1$ and, for any $\gamma$ and $\beta$ positive integers, $a_{\gamma}=\frac{(N-1)^{2\gamma}}{2^{4\gamma}}$, $b_{\gamma,\beta} = \prod_{j = 0}^{\gamma-1}\frac{(N+(\beta+4j))^2(N-(\beta+4j)-4)^2}{16}$, $d_{\gamma}$ and $e_{\gamma}$ are the constants defined in Theorem \ref{mainhigher0}.

\end{thm}

 \begin{proof}
From \eqref{poincareeq} we know 
 
 \begin{align*}
 \int_{\hn} |\nabla (\Delta^m u)|^2 \ dv_{\hn} & \geq \left( \frac{N-1}{2} \right)^{2} \int_{\hn} (\Delta^{m} u)^2 \ dv_{\hn} + \frac{1}{4} \int_{\hn} \frac{(\Delta^m u)^2}{r^2} \ dv_{\hn} \,.
 \end{align*} 
Now, by applying  \eqref{highereven2} and \eqref{yangxtended} we deduce 
$$ \int_{\hn} |\nabla (\Delta^m u)|^2 \ dv_{\hn}  \geq  \left( \frac{N-1}{2} \right)^{2}\left(\left( \frac{N-1}{2} \right)^{4(m-h)-2} \int_{\mathbb{H}^{N}}  |\nabla \Delta_{\hn}^{h} u|^{2} \ dv_{\mathbb{H}^{N}} +
  \sum_{i = 1}^{2m}\bar \alpha^{i} \int_{\mathbb{H}^{N}} \frac{u^2}{r^{2i}} \ dv_{\mathbb{H}^{N}} \right)$$
$$+ \frac{1}{4} \left( a_{m} \int_{\hn} \frac{u^2}{r^2} \ dv_{\hn} + \sum_{j = 1}^{2m -1} a^{j}_{m, 2} \int_{\hn} \frac{u^2}{r^{2j + 2}} \ dv_{\hn} + b_{m, 2} \int_{\hn} \frac{u^2}{r^{4m + 2}} \right)\,,$$
 by which the existence of $2m+1$ positive constants $\bar \delta^i=\delta^i(N,m,h)$ such that \eqref{higherodd2} holds follows. Furthermore, one has
$$\bar \delta^1:=\left( \frac{N-1}{2} \right)^{2}\bar \alpha_1+\frac{1}{4}  a_{m}  \quad \text{and} \quad  \bar\delta^{2m +1} :=\frac{1}{4} b_{m, 2} $$
 and the thesis follows by recalling the definition of $\bar \alpha_1$ in Theorem \ref{mainhighereo} for $h=m-1$ and $h\neq m-1$.
 \end{proof}

 \section{Proof of Theorem \ref{PRHinequality} and Corollary \ref{cor2}}\label{proof21}
 This section is devoted to the proofs of Theorem \ref{PRHinequality} and Corollary \ref{cor2}. The proof of Theorem \ref{PRHinequality} mainly relies on the transformation 
$u \rightarrow (\sinh r)^{\frac{(N-1)}{2}} u$, which uncovers the Poincar\'e term, and spherical harmonics technique. Before entering the proof we recall some facts on spherical harmonics.\par \medskip\par 
The Laplace-Beltrami operator on hyperbolic space in spherical coordinates is given by
\[
 \Delta_{\hn} = \frac{\partial^2}{\partial r^2} + (N-1)\, \coth r
\frac{\partial}{\partial r} + \frac{1}{\sinh ^2 r} \Delta_{\mathbb{S}^{N-1}},
\]
where $\Delta_{\mathbb{S}^{N-1}}$ is the Laplace-Beltrami operator on the unit sphere $\mathbb{S}^{N-1}.$
If we write $u(x) = u(r,\sigma) \in C_c^{\infty}(\hn),$ $r \in [0, \infty), \sigma \in \mathbb{S}^{N-1},$ then
by \cite[Ch.4, Lemma 2.18]{ES}  we have that
\[
 u(x): = u(r, \sigma) = \sum_{n= 0}^{\infty} d_{n}(r) P_{n}(\sigma)
\]
in $L^2({\hn}),$ where $\{ P_{n}\}$ is a complete orthonormal system of spherical harmonics and
\[
 d_{n}(r) = \int_{\mathbb{S}^{N-1}} u(r, \sigma) P_{n}(\sigma) \ d\sigma.
\]
We note that the spherical harmonic $P_{n}$ of order $n$ is the restriction to $\mathbb{S}^{N-1}$ of a homogeneous harmonic
polynomial of degree $n.$ Now we recall the following

\begin{lem}\label{arm} \cite[Lemma 2.1]{MSS}
 Let $P_{n}$ be a spherical harmonic of order $n$ on $\mathbb{S}^{N-1}.$ Then for every $n \in \mathbb{N}_{0}$

\[
 \Delta_{\mathbb{S}^{N-1}} P_{n} = -(n^2 + (N-2)n)P_{n}.
\]
The values $\lambda_{n} := n^2 + (N-2)n$ are the eigenvalues of the Laplace-Beltrami operator $- \Delta_{\mathbb{S}^{N-1}}$
on $\mathbb{S}^{N-1}$ and enjoy the property $\lambda_{n} \geq 0$ and $\lambda_{0} = 0.$
The corresponding eigenspace consists of all the spherical harmonics of order $n$ and has dimension $d_{n}$ where $d_{0} = 1,$
$d_{1} = N$ and
$$
 d_{n} = \left(\begin{array}{c}
N+n-1\\
n
\end{array} \right) - \left(\begin{array}{c}
N+n-3\\
n-2\
\end{array}\right),
$$
for $n\geq 2.$
\end{lem}
From Lemma \ref{arm} it is easy to see that
\begin{equation*}
 \Delta_{\hn} u(r, \sigma) = \sum_{n= 0}^{\infty} \left( d_{n}^{\prime \prime}(r) +
(N-1) \coth r d_{n} (r) - \frac{\lambda_{n} d_{n}(r)}{\sinh^2 r} \right) P_{n}(\sigma).
\end{equation*}

In the sequel we will also exploit the following 1-dimensional Hardy-type inequality from \cite{BGG}:

\begin{lem}\label{hardytype}

For all $ u \in C^{\infty}_c(0,\infty)$ there holds

\begin{equation*}
 \int_{0}^{\infty} \frac{u^{\prime 2}}{\sinh^2 r} \ dr \geq \frac{9}{4} \int_{0}^{\infty} \frac{u^2}{\sinh^4 r} \ dr
+ \int_{0}^{\infty} \frac{u^2}{\sinh^2 r} \ dr.
\end{equation*}

\end{lem}
\par \bigskip\par
{\bf{Proof of Theorem \ref{PRHinequality}.}}\par \medskip\par
The proof is divided in several steps. \par \medskip\par
{\bf{Step 1.}} For $u \in C_{c}^{\infty}(\mathbb{H}^{N})$ we define  
\[
v(x) = (\sinh r)^{\frac{N-1}{2}} u(x), \ \mbox{where}  \ r = \rho(x,x_0)
\]
Then the following relation holds for $x = (r, \sigma) \in (0, \infty) \times \mathbb{S}^{N-1},$
\begin{equation}\label{ph1}
|\nabla_{\mathbb{H}^{N}} u|^2 = (\sinh r)^{-(N-1)} \left( |\nabla_{\mathbb{H}^{N}} v|^2 +
\frac{(N-1)^2}{4} \coth^2 r v^2  - (N-1) \coth r v \frac{\partial v}{\partial r} \right).
\end{equation}
\\

{\bf{Proof of Step 1.}} 
 A straightforward computation gives
\begin{align*}
|\nabla_{\mathbb{H}^{N}} v|^2& = \left(\frac{\partial v}{\partial r}\right)^2+\frac{1}{\sinh^2 r}|\nabla_{\mathbb{S}^{N}}v|^2
\notag\\
& = \frac{(N-1)^2}{4} (\sinh r)^{N-1} \coth^2 r u^2 + 
|\nabla_{\mathbb{H}^{N}} u|^2 (\sinh r)^{N-1} \notag \\
& + (N-1) (\sinh r)^{N-1} \coth r (\sinh r)^{\frac{-(N-1)}{2}} v \frac{\partial u}{\partial r}  \notag \\
& =  \frac{(N-1)^2}{4}  \coth^2 r v^2 + 
|\nabla_{\mathbb{H}^{N}} u|^2 (\sinh r)^{N-1} - \frac{(N-1)^2}{2} \coth^2 r v^2 \notag \\
& + (N-1) \coth r v \frac{\partial v}{\partial r}.
\end{align*}
Now by rearranging the terms above we conclude the proof of Step 1. \par \medskip\par
{\bf{Step 2.}}
In this step we compute,  
\begin{equation*}\label{transformation}
 \Delta_{\hn} v = \left( \frac{\partial^2}{\partial r^2} + (N-1) \coth r \frac{\partial}{\partial r} +
\frac{1}{\sinh^2} \Delta_{\mathbb{S}^{N-1}}\right) (\sinh r)^{\frac{N-1}{2}} u.
\end{equation*}

\begin{align*}
  \Delta_{\hn} v& = \frac{(N-1)(N-3)}{4} (\sinh r)^{\frac{N-1}{2}} \coth^2 r \ u + (N-1) (\sinh r)^{\frac{N-3}{2}} \cosh r
\frac{\partial u}{\partial r} \\
& + \frac{(N-1)}{2} (\sinh r)^{\frac{N-1}{2}} u + (\sinh r)^{\frac{N-1}{2}} \frac{\partial^2 u}{\partial r^2} \\
& + \frac{(N-1)^2}{2} \coth^2 r (\sinh r)^\frac{N-1}{2} u + (N-1) \coth r (\sinh r)^{\frac{N-1}{2}} \frac{\partial u}{\partial r} \\
& + (\sinh r)^{\frac{N-1}{2}} \frac{1}{\sinh^2 r} \Delta_{\mathbb{S}^{N-1}} u\\
& = (\sinh r)^{\frac{N-1}{2}} \left[ \frac{\partial^2 u}{\partial r^2} + (N-1) \coth r \frac{\partial u}{\partial r}
+ \frac{1}{\sinh^2 r} \Delta_{\mathbb{S}^{N-1}} u \right] \\
& + \left[ \frac{(N-1)(N-3)}{4} + \frac{(N-1)^2}{2} \right] \coth^2 r (\sinh r)^{\frac{N-1}{2}} u +
\frac{(N-1)}{2} (\sinh r)^{\frac{N-1}{2}} u \\
&+ (N-1) (\sinh r)^{\frac{N-3}{2}} \cosh r \left[ \frac{\partial}{\partial r} ((\sinh r)^{- \frac{(N-1)}{2}} v) \right]\\
& = (\sinh r)^{\frac{N-1}{2}} (\Delta_{\hn} u) +  \left[ \frac{(N-1)(N-3)}{4} + \frac{(N-1)^2}{2} \right] \coth^2 r v \\
& + \frac{(N-1)}{2} v - \frac{(N-1)^2}{2} \coth^2 r v + (N-1) \coth r \frac{\partial v}{\partial r} \\
& = (\sinh r)^{\frac{N-1}{2}} (\Delta_{\hn} u) + \frac{(N-1)(N-3)}{4} \coth^2 r v + \frac{(N-1)}{2} v + (N-1) \coth r
\frac{\partial v}{\partial r}\,.
\end{align*}
Hence, we get
\begin{equation}
 \begin{aligned}\label{trans1}
  \Delta_{\hn} u  &= \frac1{(\sinh r)^{\frac{(N-1)}{2}}} \left[ \Delta_{\hn} v -
 \left( \frac{(N-1)(N-3)}{4} \coth^2 r + \frac{(N-1)}{2} \right) v \right.  \left. -(N-1)\coth r \frac{\partial v}{\partial r} \right] \\
&= \frac1{(\sinh r)^{\frac{(N-1)}{2}}} \left[ \frac{\partial^2 v}{\partial r^2} - \left( \frac{(N-1)(N-3)}{4} \coth^2 r
+ \frac{(N-1)}{2} \right) v \right. \left. + \frac{1}{\sinh^2 r} \Delta_{\mathbb{S}^{N-1}} v \right].
 \end{aligned}
 \end{equation}

\par \bigskip\par
{\bf{ Step 3.}} Expanding $v$ in spherical harmonics 
\[
v(x) : = v(r, \sigma ) =  \sum_{n=0}^{\infty} d_{n}(r) P_{n}(\sigma)
\] 
and observing that 
\[
\int_{\mathbb{H}^{N}} (\sinh r)^{-(N-1)}  |\nabla_{\mathbb{H}^{N}} v|^2 \ dv_{\mathbb{H}^{N}} = 
\sum_{ n = 0}^{\infty} \int_{0}^{\infty}  \left( (d_{n}^{\prime}(r))^2 + \lambda_{n} \frac{d_{n}^2(r)}{\sinh^2 r} \right) \ dr,
\]
and substituting in \eqref{ph1}, we infer 
\begin{align}\label{ph4}
\int_{\mathbb{H}^{N}}  |\nabla_{\mathbb{H}^{N}} u|^2 \ dv_{\mathbb{H}^{N}} 
 = &\notag \\\sum_{n = 0 }^{\infty}   \int_{0}^{\infty}  &  \left( (d_{n}^{\prime}(r))^2 + 
\lambda_{n} \frac{d_{n}^2(r)}{\sinh^2 r} + \frac{(N-1)^2}{4} \coth^2 r d_{n}^2(r)- (N-1) \coth r d_{n}(r) d_{n}^{\prime}(r)  \right)  dr.
\end{align}

Further expanding in spherical harmonics and putting this in \eqref{trans1}, we have
\begin{align}
 \int_{\mathbb{H}^{N}}  |\Delta_{\hn} u|^2  \ dv_{\mathbb{H}^{N}} 
 =& \notag \\ \sum_{n=0}^{\infty} \int_{0}^{\infty} &\left( d_{n}^{\prime \prime}(r) - \frac{(N-1)(N-3)}{4} \coth^2 r d_{n}(r) - \frac{(N-1)}{2} d_{n}(r) - \frac{\lambda_{n}}{\sinh^2 r} d_{n}(r) \right)^2 dr, 
\end{align}
where the eigenvalues $\lambda_{n}$ are repeated according to their multiplicity. Let us write using \eqref{ph1} and \eqref{ph4}, 

\begin{align*}
& \int_{\mathbb{H}^{N}} (\Delta_{\mathbb{H}^{N}} u)^2 \ dv_{\mathbb{H}^{N}} - \left( \frac{N-1}{2} \right)^2 \int_{\mathbb{H}^{N}} |\nabla_{\mathbb{H}^{N}} u|^2 
\ dv_{\mathbb{H}^{N}}   = \sum_{n=0}^{\infty} \left[ \int_{0}^{\infty}   \left( d_{n}^{\prime \prime}(r) - \frac{(N-1)(N-3)}{4} \coth^2 r d_{n}(r) \right. \right. \notag \\
&  \left. \left. - \frac{(N-1)}{2} d_{n}(r) - \frac{\lambda_{n}}{\sinh^2 r} d_{n}(r) \right)^2 \ dr - \left( \frac{N-1}{2} \right)^2    \int_{0}^{\infty} 
  \left( (d_{n}^{\prime}(r))^2 + \lambda_{n} \frac{d_{n}^2(r)}{\sinh^2 r} + \frac{(N-1)^2}{4} \coth^2 r d_{n}^2(r)  \right. \right. \notag \\
&  \left. \left.   - (N-1) \coth r d_{n}(r) d_{n}^{\prime}(r)  \right) \right]  \  dr.
\end{align*}

 Considering each term separately and simplifying further, for detail see the proof of \cite[Theorem 3.1]{BGG}), we get

\begin{align}\notag \label{estimate1}
&\int_{0}^{\infty} \left( d_{n}^{\prime \prime}(r) - \frac{(N-1)(N-3)}{4} \coth^2 r d_{n}(r)
 - \frac{(N-1)}{2} d_{n}(r) - \frac{\lambda_{n}}{\sinh^2 r} d_{n}(r) \right)^2 \ dr \\ \notag 
 & = \int_{0}^{\infty} (d_{n}^{\prime \prime}(r))^2 \ dr + \frac{(N-1)^2}{2}  \int_{0}^{\infty} (d_{n}^{\prime}(r))^2 \ dr  + \left( \frac{(N-1)(N-3)}{2} +
2 \lambda_{n} \right)  \int_{0}^{\infty} \frac{1}{\sinh^ 2 r}  (d_{n}^{\prime}(r))^2 \ dr   \notag \\
& \frac{(N-1)^4}{16} \int_{0}^{\infty} (d_{n}(r))^2 \ dr + \left( \lambda_{n}^2 + \frac{(N-1)(N-3)}{2} \lambda_{n} - 6 \lambda_{n} + \frac{(N-1)^2 (N-3)^2}{16} \right. \notag \\
& \left. - \frac{3}{2} (N-1)(N-3) \right) \int_{0}^{\infty} \frac{1}{\sinh^4 r} (d_{n}(r))^2 \ dr  + \left( \frac{(N-1)^2(N-3)^2}{8} + \frac{(N-1)^2(N-3)}{4}  \right. \notag \\
& \left. + \frac{(N-1)(N-3)}{2} \lambda_{n} + (N-5) \lambda_{n} - (N-1)(N-3) \right) \int_{0}^{\infty} \frac{1}{\sinh^2 r} (d_{n}(r))^2 \ dr \notag \\
\end{align}
and 

\begin{align}\label{estimate2}
  \left(\frac{N-1}{2} \right)^2 \int_{\mathbb{H}^{N}} |\nabla_{\mathbb{H}^{N}} u|^2 \ dv_{\mathbb{H}^{N}} & =   \left( \frac{N-1}{2} \right)^2 \int_{0}^{\infty} (d_{n}^{\prime}(r))^2 \ dr -  \frac{(N-1)^4}{16}  \int_{0}^{\infty} (d_{n}(r))^2 \ dr \notag \\
&  - \frac{(N-1)^3 (N - 3)}{16} \int_{0}^{\infty}
 \frac{1}{\sinh^2 r} (d_{n}(r))^2 \ dr.
\end{align}
By \eqref{estimate1} and \eqref{estimate2}, using Lemma \ref{hardytype}, the 1-dimensional Hardy inequality:
\begin{equation*}
\int_{0}^{\infty}  (d_{n}^{\prime}(r))^2 \ dr \geq \frac{1}{4} \int_{0}^{\infty} \frac{d_{n}^{2}(r)}{r^2} \ dr
\end{equation*}
and the 1-dimensional Rellich inequality:
\[
 \int_{0}^{\infty} (d_{n}^{\prime \prime}(r))^2 \ dr \geq \frac{9}{16} \int_{0}^{\infty} \frac{d_{n}^2 (r)}{r^4} \ dr ,
\] 
we obtain
\begin{align*}
& \int_{\mathbb{H}^{N}} (\Delta_{\mathbb{H}^{N}} u)^2 \ dv_{\mathbb{H}^{N}} - \left( \frac{N-1}{2} \right)^2 \int_{\mathbb{H}^{N}} |\nabla_{\mathbb{H}^{N}} u|^2 
\ dv_{\mathbb{H}^{N}}  \notag  \\
& \geq \frac{9}{16} \int_{0}^{\infty} \frac{d_{n}^{2}(r)}{r^4} \ dr + \frac{(N-1)^2}{16} \int_{0}^{\infty} \frac{d_{n}^{2}(r)}{r^2 } \ dr + A_{N} \int_{0}^{\infty} \frac{d_{n}^{2}(r)}{\sinh^4 r} \ dr +
 B_{N} \int_{0}^{\infty} \frac{d_{n}^{2}(r)}{\sinh^4 r} \ dr,
\end{align*}
where
\[
 A_{n} = \left[ \lambda_{n}^2 +  \frac{N(N-4)}{2} \lambda_{n} +
\frac{((N-1)(N-3))^2}{16} - \frac{3}{8} (N-1)(N-3) \right]
\]
and
\[
B_{n} = \left[   \frac{(N^2 - 2N - 5)}{4} \lambda_{n} + \frac{(N-1)^2(N-3)}{4} + \frac{(N-1)^2(N-3)(N-5)}{16} - \frac{(N-1)(N-3)}{2} \right].
\]
We note that
\[
\min_{n \in \mathbb{N}_{0}} A_{n}=\frac{(N-1)(N-3)(N^2-4N-3)}{16} \quad \mbox{and} \quad \min_{n \in \mathbb{N}_{0}} B_{n}= \frac{(N-1)(N-3)(N^2 - 2N - 7)}{16}
\]
for $N \geq 5$ and hence they are both positive. Also we have
\[
 \int_{\hn} u^2 \ dv_{\hn} = \int_{\hn} v^2 (\sinh r)^{-(N-1)} \ dv_{\hn} = \sum_{n= 0}^{\infty} \int_{0}^{\infty} d_{n}^2(r) \ dr,
\]
similarly
\[
 \int_{\hn} \frac{u^2}{r^2} \ dv_{\hn} = \sum_{n=0}^{\infty} \int_{0}^{\infty} \frac{d_{n}^2(r)}{r^2} \ dr
\]
and so on. Now, using all these facts, we obtain
\begin{align*}
 \int_{\hn} (\Delta_{\hn} u)^2 \ dv_{\hn} - \left(\frac{N-1}{2} \right)^2 \int_{\hn} |\nabla_{\hn} u|^2 \ dv_{\hn}
\geq \frac{9}{16} \int_{\hn} \frac{u^2}{r^4} \ dv_{\hn} + \frac{(N-1)^2}{16} \int_{\hn} \frac{u^2}{r^2} \ dv_{\hn}\\
\frac{(N-1)(N-3)(N^2-4N-3)}{16} \int_{0}^{\infty} \frac{u^2}{\sinh^4 r} \ dv_{\hn} +\frac{(N-1)(N-3)(N^2 -2N -7)}{16} \int_{0}^{\infty} \frac{u^2}{\sinh^2 r} \ dv_{\hn},
\end{align*}
and hence the proof of inequality \eqref{npoincare}. 

 \par \medskip\par
{\bf {Step 4.}}  Next we show the optimality of the constant $\frac{(N-1)^2}{16}.$  Let us suppose that the $\frac{(N-1)^2}{16}$ is not optimal, i.e., there exist $ C >  \frac{(N-1)^2}{16}$ such that there holds,
 \begin{align*}
 \int_{\hn} (\Delta_{\hn} u)^2 \ dv_{\hn} - \left(\frac{N-1}{2} \right)^2 \int_{\hn} |\nabla_{\hn} u|^2 \ dv_{\hn}
\geq  C \int_{\hn} \frac{u^2}{r^2} \ dv_{\hn},
 \end{align*}
 using \cite[Theorem 2.1]{BGG}) and above we obtain
 
 \begin{align}\label{estimate3}
 \int_{\mathbb{H}^{N}} (\Delta_{\mathbb{H}^{N}} u)^2 \ dv_{\mathbb{H}^{N}} & \geq  C  \int_{\mathbb{H}^{N}} \frac{u^2}{r^2} \ dv_{\mathbb{H}^{N}} 
  + \frac{(N-1)^2}{4} \left[ \frac{(N-1)^2}{4} \int_{\mathbb{H}^{N}} u^2 \ dv_{\hn} + \frac{1}{4} \int_{\hn} \frac{u^2}{r^2} \ dv_{\hn} \right] \notag \\
 & =\left( C + \frac{(N-1)^2}{16} \right)  \int_{\mathbb{H}^{N}} \frac{u^2}{r^2} \ dv_{\mathbb{H}^{N}} 
 + \frac{(N-1)^4}{16} \int_{\hn} u^2 \ dv_{\mathbb{H}^{N}},
 \end{align}
 comparing \eqref{estimate3} with \cite[Theorem 3.1]{BGG}, we conclude that $C \leq \frac{(N-1)^2}{16}$ which gives a contradiction and hence $\frac{(N-1)^2}{16}$ is the best constant.
 $\Box$
\par \bigskip\par
{\bf{Proof of Corollary 2.2. }}
 By considering the upper half space model $\mathbb{R}^{N}_{+} $ for $\hn$ and using the explicit expression of the gradient in these coordinates, namely $\nabla_{\mathbb{H}^{N}} = y^2 \nabla,$ we obtain

\begin{align}\label{pf1}
\int_{\hn} |\nabla_{\hn} u(x,y)|^2 \ dv_{\hn}  & = \int_{\mathbb{R}^{+}} \int_{\mathbb{R}^{N-1}}  y^{2\alpha + 2 - N} |\nabla v|^2 \ dx  \ dy 
+ \alpha^2 \int_{\mathbb{R}^{+}} \int_{\mathbb{R}^{N-1}} y^{2 \alpha - N} v^2 \ dx \ dy \notag \\
& - \alpha (2 \alpha + 1 - N) \int_{\mathbb{R}^{+}} \int_{\mathbb{R}^{N-1}} y^{2 \alpha - N} v^2 \  dx \ dy
\end{align}
and also using Laplacian expression,  $\Delta_{\mathbb{H}^{N}} = y^2 \Delta - (N-2) y$, we get 

\begin{equation}\label{pf2}
\Delta_{\mathbb{H}^{N}} u = y^{\alpha + 2} \Delta v + (2 \alpha - (N-2)) y^{\alpha} \frac{\partial v}{\partial y}  + \alpha (\alpha - (N-1)) y^{\alpha} v,
\end{equation}
where  and $u (x,y) : = y^{\alpha} v(x,y).$
With $\alpha = \frac{N-2}{2}, $ we get

\begin{align*}
 \int_{\hn} (\Delta_{\hn} u(x,y))^2 \ dv_{\hn} & = \int_{\mathbb{R}^{+} \times \mathbb{R}^{N-1}} \left( y^{N+1} (\Delta v)^2 + \frac{N^2(N-2)^2}{16} y^{N-2} v^2 
- \frac{N(N-2)}{2} y^{N} v \Delta v \right) \frac{dx \ dy}{y^{N}}   \\ 
& = \int_{\mathbb{R}^{N-1} \times \mathbb{R}^{+}} \left( y^{2} (\Delta v)^2 + \frac{N^2(N-2)^2}{16} \frac{v^2}{y^2} 
+ \frac{N(N-2)}{2} y^{N} |\nabla v|^2  \right) dy \ dx.
\end{align*}

Similarly with $\alpha = \frac{N-4}{2}, $ and by denoting $\frac{\partial v}{\partial y} := v_{y},$  we get

\begin{align*}
\int_{\hn} (\Delta_{\hn} u(x,y))^2 \ dv_{\hn} & = \int_{\mathbb{R}^{+} \times \mathbb{R}^{N-1}} \left( (\Delta v)^2 + 4 \frac{v_{y}^2}{y^2} 
+ \frac{(N-4)^2 (N +2)^2}{16} \frac{v^2}{y^4} \right.   \\ 
 & \left. - 4 \frac{v_{y} \Delta v}{y} - \frac{(N-4)(N + 2)}{2} \frac{v \Delta v}{y^2} + (N-4)(N+2) \frac{v v_{y}}{y^3} \right) dy \ dx   \\
 & = \int_{\mathbb{R}^{+} \times \mathbb{R}^{N-1}} \left( (\Delta v)^2 + 4 \frac{v_{y}^2}{y^2} + \frac{(N-4)^2 (N+2)^2}{16} \frac{v^2}{y^4}  \right.   \\ 
 & \left. - 4 \frac{v_{y}^2}{y^2} + 2 \frac{|\nabla v|^2}{y^2} + \frac{(N-4)(N+2)}{2} \frac{|\nabla v|^2}{y^2} \right) dy \ dx \\
 & =  \int_{\mathbb{R}^{+} \times \mathbb{R}^{N-1}} \left( (\Delta v)^2 + \frac{(N-4)^2 (N+2)^2}{16} \frac{v^2}{y^4}  
   + 2 \frac{|\nabla v|^2}{y^2} \right.  \\
   & \left.  + \frac{(N-4)(N+2)}{2} \frac{|\nabla v|^2}{y^2} \right) dy \ dx 
 \end{align*}
 
 Now the proof follows by inserting \eqref{pf1}, \eqref{pf2} and the above computations in \eqref{npoincare} with $\alpha = \frac{N-2}{2}$ and $\alpha = \frac{N-4}{2}$
successively. \\

Next we turn to the optimality issues. Assume by contradiction that  the following inequality holds $$\int_{\mathbb{R}^{+}} \int_{\mathbb{R}^{N-1}} \left( y^2 (\Delta v)^2 +
c |\nabla v|^2 \right) \ dx \ dy
 \geq \frac{N(N-2)}{16} \int_{\mathbb{R}^{+}} \int_{\mathbb{R}^{N-1}} \frac{v^2}{y^2} \ dx \ dy $$
for all $ u \in C^{\infty}_{c}(\mathbb{H}^{N})$ with $c< \frac{N^2 -2N -1}{4}$. The above inequality, jointly with \eqref{high} with $k = 1,$  $l = 0$ and Hardy-Maz'ya inequality:

\[
\int_{\mathbb{R}^{+}} \int_{\mathbb{R}^{N-1}} |\nabla v|^2 \ dx \ dy \geq \frac{1}{4} \int_{\mathbb{R}^{+}} \int_{\mathbb{R}^{N-1}} \frac{v^2}{y^2} \ dx \ dy,
\]
where $\frac{1}{4}$ is the best constant, see \cite{Ma} and \cite{FMT, FTT}, yields
$$\int_{\mathbb{H}^{N}} |\Delta_{\mathbb{H}^{N}} u|^2 \ dv_{\mathbb{H}^{N}}\geq\frac{(N -1)^2}{4} \int_{\mathbb{H}^{N}} |\nabla_{\hn}u|^2 \ dv_{\mathbb{H}^{N}}+\left(\frac{N^2-2N-1}{4}-c\right)  \int_{\mathbb{R}^{+}} \int_{\mathbb{R}^{N-1}} |\nabla v|^2 \ dx \ dy$$
$$\geq\frac{(N -1)^4}{16} \int_{\mathbb{H}^{N}} u^2 \ dv_{\mathbb{H}^{N}}+\frac{1}{4}  \left(\frac{N^2-2N-1}{4}-c\right)  \int_{\mathbb{R}^{+}} \int_{\mathbb{R}^{N-1}} \frac{v^2}{y^2} \ dx \ dy$$
$$=\left(\frac{(N -1)^4}{16}+\frac{1}{4}  \left(\frac{N^2-2N-1}{4}-c\right) \right) \int_{\mathbb{H}^{N}} u^2 \ dv_{\mathbb{H}^{N}}\,,$$
a contradiction with \eqref{high} with $k = 2$ and $l = 0$. The optimality of the other constants follows straightforwardly from what remarked above. $\Box$

\par\bigskip\noindent
\textbf{Acknowledgments.} We are grateful to G. Grillo for fruitful discussions during the preparation of the manuscript. The authors are partially supported by the Research Project FIR (Futuro in Ricerca) 2013 \emph{Geometrical and qualitative aspects of PDE's}. The first author is member of the Gruppo Nazionale per l'Analisi Matematica, la Probabilit\`a e le loro Applicazioni (GNAMPA) of the Istituto Nazionale di Alta Matematica (INdAM).

\end{document}